\documentclass[reqno,english]{amsart}
\usepackage{etex}
\usepackage{amsmath,amssymb,amsthm,bbm,mathtools,comment}
\usepackage[shortlabels]{enumitem}
\usepackage[pdftex,colorlinks,backref=page,citecolor=blue]{hyperref}
\usepackage[mathscr]{euscript}
\usepackage{tikz,babel,adjustbox}
\usepackage{cmtiup}
\usepackage{amsfonts}
\usepackage{graphicx}
\usepackage{caption}
\usepackage{verbatim}
\usepackage{array}
\usepackage[frame,cmtip,arrow,matrix,line,graph,curve]{xy}
\usepackage{graphpap, color, pstricks}
\usepackage{pifont}
\usepackage{cmtiup}

\allowdisplaybreaks

\allowdisplaybreaks

\theoremstyle{plain}
\newtheorem{theorem}{Theorem}[section]
\newtheorem{lemma}[theorem]{Lemma}

\newtheorem{prop}[theorem]{Proposition}

\newtheorem{mydef}[theorem]{Definition}
\newtheorem{remark}[theorem]{Remark}
\theoremstyle{remark}

\newtheorem{obs}[theorem]{Observation}

\newcommand{\pr}{\mathbb{P}}
\newcommand{\RR}[0]{{\mathbb R}}

\newcommand{\EE}[0]{\mathbb{E}}

\newcommand{\WW}[0]{\mathbf{W}}
\newcommand{\ZZ}[0]{\mathbf{Z}}

\newcommand{\sss}[0]{\mathbf{s}}

\newcommand{\sub}[0]{\subseteq}

\renewcommand{\dots}[0]{,\ldots,}

\newcommand{\beq}[1]{\begin{equation}\label{#1}}
\newcommand{\enq}[0]{\end{equation}}

\newcommand{\gl}[0]{\lambda }

\newcommand{\mn}[0]{\medskip\noindent}
\newcommand{\nin}[0]{\noindent}

\newcommand{\cF}{\mathcal{F} }
\newcommand{\cG}{\mathcal{G} }
\newcommand{\cH}{\mathcal{H} }

\newcommand{\cM}{\mathcal{M} }

\newcommand{\cT}{\mathcal{T} }
\newcommand{\cU}{\mathcal{U} }

\newcommand{\bgl}{\boldsymbol\gl }

\title{On a conjecture of Talagrand on selector processes and \\ a consequence on positive empirical processes}
\author{Jinyoung Park \and Huy Tuan Pham}
\email{jinyoungpark@nyu.edu}
\address{Department of Mathematics, Courant Institute of Mathematical Sciences, New York University. 
251 Mercer St, New York, NY 10012}
\email{huypham@stanford.edu}
\address{Department of Mathematics, Stanford University\\
450 Jane Stanford Way, Building 380, Stanford, CA 94305}

\begin{document}

\maketitle
\begin{abstract}
For appropriate Gaussian processes, as a corollary of the majorizing measure theorem, Michel Talagrand (1987) proved that the event that the supremum is significantly larger than its expectation can be covered by a set of half-spaces whose sum of measures is small. We prove a conjecture of Talagrand that is the analog of this result in the Bernoulli-$p$ setting, and answer a question of Talagrand on the analogous result for general positive empirical processes. 
\end{abstract}
\maketitle

\section{Introduction}\label{Intro}

The study of suprema of stochastic processes is of central interest in probability theory, with influential applications in related areas. We refer the readers to \cite{Talagrand0, Talagrand2} for extensive discussions of various aspects of this subject. Through many fundamental developments, one now has fairly good understanding of the suprema of centered Gaussian processes\footnote{Following \cite{Talagrand2}, we always assume Gaussian processes are centered, i.e., $\mathbb EZ_t=0$ for all $t \in T$.}.  In particular, one can associate each Gaussian process $(Z_t)_{t\in T}$ indexed by a set $T$ with a metric on $T$ given by $d(t,s):=(\EE[(Z_t-Z_s)^2])^{1/2}$, and Talagrand's celebrated generic chaining bound and majorizing measure theorem \cite{TalagrandCB, TalagrandCBs, Talagrand2} determine the expectation of the supremum $\sup_{t\in T}Z_t$ (up to a constant factor) by a quantity depending only on the metric space $(T,d)$. Via this fundamental result, one can obtain deep insights and characterizations of the suprema of Gaussian processes. One important example is Theorem \ref{thm:gaussian} below, 
which gives a nice geometric characterization of large suprema of Gaussian processes: such event must be contained in a union of half-spaces whose sum of measures is small.

\begin{theorem}[Talagrand, Theorem 2.12.2 in \cite{Talagrand2}] \label{thm:gaussian}
There exists $L>0$ such that the following holds. Let $g$  be an $M$-dimensional standard Gaussian vector. For $\mathcal T \sub \mathbb{R}^M$, consider the process $Z_t = \langle t, g\rangle$ for $t\in \cT$. Then one can find a sequence of half-spaces $H_k$ of $\mathbb{R}^M$ with 
\[
\left\{\sup_{t\in \cT}Z_t \ge L\mathbb{E}\sup_{t\in \cT}Z_t\right\} \subset \bigcup_{k \ge 1} H_k,
\]
and
\[
\sum_{k\ge 1} \mathbb{P}(H_k) \le \frac{1}{2}. 
\]
\end{theorem}

Our main contribution in this paper is the proof of a conjecture of Talagrand on selector processes (Theorem \ref{Conj 5.7}; originally \cite[Problem 4.1]{Talagrand06}, \cite[Conjecture 5.7]{Talagrand} and \cite[Research Problem 13.2.3]{Talagrand2}) and a result on positive empirical processes (Theorem \ref{thm:pos-emp}; a question of Talagrand \cite{Talagrand-per} and a problem posed in \cite{Talagrand06}), which are analogous to Theorem \ref{thm:gaussian}. We first quickly state our main results, and then provide more context, definitions, and motivations for Talagrand's questions.

Given a finite set $X$, write $2^X$ for the power set of $X$. For $p \in [0,1]$, let $\mu_p$ be the product measure on $2^X$ given by $\mu_p(A)=p^{|A|}(1-p)^{|X \setminus A|}$. 
We use $X_p$ for the random variable whose distribution is $\mu_p$. For $S \sub X$, define the \textit{upset generated by $S$} to be $\langle S \rangle := \{T: T \supseteq S\}$. Following \cite{Talagrand}, we say $\cF \sub 2^X$ is \textit{$p$-small} if there is $\cG \sub 2^X$ such that 
\beq{cover1} \cF \sub \langle \cG \rangle :=\bigcup_{S \in \cG}\langle S \rangle\enq
and
\beq{p-small} \sum_{S \in \cG} p^{|S|} \le 1/2.\enq
We say $\cG$ is a \textit{cover} of $\cF$ if \eqref{cover1} holds.

Our first main result is the Bernoulli-$p$ analog of Theorem \ref{thm:gaussian}.

\begin{theorem} \label{Conj 5.7} 
There exists $L>0$ such that the following holds. Consider any $0<p< 1$, any finite set $X$ and any collection $\Lambda$ of sequences $\bgl=(\gl_i)_{i \in X}$ with $\gl_i \ge 0$. 
Then the family
\[ \left\{S \sub X:\sup_{\bgl \in \Lambda} \sum_{i \in S} \gl_i \ge L\EE \sup_{\bgl \in \Lambda} \sum_{i \in X_p} \gl_i\right\}\]
is $p$-small.

\end{theorem}

\nin In \cite{Talagrand}, Talagrand explains the meaning of the above theorem this way: Conjecture 5.7 (now Theorem~\ref{Conj 5.7}) shows that ``if you are given a selector process, and would like to prove that, within a multiplicative factor, the quantity $\EE \sup_{\bgl \in \Lambda} \sum_{i \in X_p} \gl_i \le M$ for a constant $M$, there is in the end no other way than to find the witnesses that the set $\{S \sub X:\sup_{\bgl \in \Lambda} \sum_{i \in S} \gl_i \ge LM\}$ is small.'' In the same place, Talagrand suggests that this result ``provides fundamental information.'' 

Our second main result is the analog of Theorem \ref{thm:gaussian} for positive empirical processes (see \eqref{empirical} and the discussion that follows). 

\begin{theorem}\label{thm:pos-emp}
There exists $L>0$ such that the following holds. For any $N>0$ and i.i.d. random variables $Y_1 \dots Y_N$ distributed according to a Borel probability measure $\nu$ on a Polish space $\mathbb{T}$, and any finite collection $\cF$ of Borel functions $f: \mathbb{T} \to \mathbb{R}_{\ge 0}$ with $\cF \subseteq L^\infty(\mathbb{T})$, consider the positive empirical process $Z_f = \frac{1}{N} \sum_{i=1}^{N} f(Y_i)$. 
Assume that $0<\EE[\sup_{f\in \cF}Z_f]<\infty$. Then one can find a collection $\mathcal{C}$ of pairs $(g,t)$ where $g$ is a nonnegative function on $\mathbb{T}$ and $t>0$, so that with $E_{g,t} := \{Z_g \ge t\}$, we have 
\[
\left\{\sup_{f\in \cF}Z_f \ge L\mathbb{E}\sup_{f\in \cF}Z_f\right\} \subset \bigcup_{(g,t)\in \mathcal{C}} E_{g,t},
\]
and
\[
\sum_{(g,t)\in \mathcal{C}} \mathbb{P}(E_{g,t}) \le \frac{1}{2}. 
\]

\end{theorem}

{We remark that the conclusion of Theorem \ref{thm:pos-emp} readily extends to cases where $\cF$ is not necessarily finite, for example when $\cF$ is a totally bounded infinite subset of $L^\infty(\mathbb{T})$.} 

In \cite{Talagrand06}, Talagrand proved versions of our main results for the special case where the class of functions consists of indicator of sets, and posed the question of extending the results to general classes of functions as important open problems. This is fully addressed by our results in both the setting of selector processes and empirical processes. (Talagrand's result on empirical processes in \cite{Talagrand06} is stated slightly differently; see the remark at the end of Section \ref{sec.Thm1.3} for how to obtain from our proof a stronger version of Talagrand's result in the general setting.) 

\mn \textbf{More context and definitions.} A Gaussian process can be described in the form $Z_t = \sum_{i=1}^{\infty}\xi_i t_i$ where $\xi_1,\xi_2,\ldots$ are i.i.d. standard Gaussian random variables, and $t=\{t_i\}_{i \ge 1} \in T \subseteq \ell^2$ is a square-summable sequence. Alternatively, one can view Gaussian processes as random series (with i.i.d. Gaussian coefficients) of functions $f_i:T\to \RR$ (where $f_i(t) = t_i$), an object of natural interest. Generalizing the coefficients beyond the Gaussian case immediately leads to substantially more challenging questions. In particular, in the case where the coefficients $\xi_i$ are independent Rademacher random variables, a longstanding conjecture of Talagrand (``the Bernoulli Conjecture'') suggests a precise way to control the supremum in expectation in the spirit of chaining. The conjecture was only resolved recently in a breakthrough by Bednorz and Lata\l a \cite{BLatala}.

 The problem is even harder when $\xi_i$ are centered Bernoulli-$p$ random variables  {(to be contrasted with our focus later on the ordinary nonnegative Bernoulli-$p$ random variables)}. In the ``generalized Bernoulli Conjecture,'' which now is \cite[Theorem 11.12.1]{Talagrand2} and whose proof is inspired by work of Bednorz and Martynek \cite{BM}, Talagrand showed that suprema of centered Bernoulli processes can be described in terms of quantities depending only on the metric structure of appropriate classes of functions, together with quantities depending on suprema of processes of the form $Z_t = \sum_{i=1}^{M}\xi_i t_i$ for $t:\mathbb{N}\to \RR_{\ge 0}$, and $\xi_i$ i.i.d. (ordinary) Bernoulli-$p$ random variables. This way, the study of random sums of functions leads us naturally to \textit{positive selector processes}, which we now define formally. Recall that $X_p$ is the random variable whose distribution is $\mu_p$. Given a collection $\Lambda$ of sequences $\bgl=(\gl_i)_{i \in X}$, we define the \emph{selector process} associated to $\Lambda$ as the process indexed by $\Lambda$ whose value at $\lambda$ is given by 
\beq{emp}Z_\lambda := \sum_{i\in X_p}\lambda_i. \enq
A selector process is \emph{positive} if $\gl_i \ge 0$ for all $\gl \in \Lambda$ and $i\in X$. 

Next, we motivate the study of \textit{positive empirical processes}. Given i.i.d. random variables $Y_1 \dots Y_N$ distributed according to a Borel probability measure on a Polish space $\mathbb{T}$, and a class $\cF$ of functions $f:\mathbb{T}\to \mathbb{R}$, 
an \emph{empirical process} is a process indexed by $\cF$ of the form \beq{empirical} Z_f := \frac{1}{N}\sum_{i\le N}f(Y_i).\enq We say that an empirical process is \emph{positive} if the functions $f$ in the class $\cF$ are nonnegative. Empirical processes and their suprema form an important subject in probability theory and have a wide range of applications in computer science, statistics and machine learning \cite{Talagrand1, Talagrand2, SS, S}. There, one is often interested in the suprema of the empirical process, $\sup_{f\in \cF} Z_f$. While one is often interested in \emph{centered} empirical processes, i.e. those indexed by zero mean functions, by a deep result of Talagrand \cite[Theorem 6.8.3]{Talagrand2} (``the fundamental theorem of empirical processes''), the supremum of general centered empirical processes can always be controlled, in a precise sense, by quantities depending only on the metric structure of an appropriate class of functions, and the supremum of an appropriate \textbf{positive} empirical process. Thus, the study of positive empirical processes is key to understanding centered empirical processes.

As we have discussed in the previous paragraphs, the study of suprema of centered stochastic processes naturally leads us to the study of their positive counterpart. {While substantial advances in chaining allow us to understand precisely the reduction from centered processes to positive processes, suprema of positive processes are much less understood. As Talagrand explains in \cite{Talagrand2}, in many cases, we know that ``chaining explains all the boundedness due to cancellation, but what could we ask about boundedness of processes where no cancellation occurs?'' Thus, while we have good understanding of the effect of cancellation on the suprema of stochastic processes, positive processes (where there is no cancellation to exploit) are much less understood and are essentially the last missing piece in the picture. In this context, our main theorems address the task of filling in this missing piece:}
 Theorem \ref{Conj 5.7} resolves a conjecture of Talagrand on large suprema of positive selector processes, one of the questions in \cite[Chapter 13]{Talagrand2} on ``Unfulfilled dreams;'' and Theorem \ref{thm:pos-emp} answers a question of Talagrand on general positive empirical processes. We point out that our proof of Theorem \ref{thm:pos-emp} builds on a close connection between positive empirical processes and a version of selector process with multiplicities that has been informally observed in \cite{Talagrand}.

Roughly speaking, Theorem \ref{thm:gaussian} shows that, for a Gaussian process, one can find simple geometric ``witnesses'' (half-spaces) which cover the event that the supremum of the process is large, and the sum of measures of these witnesses is small. In particular, even though it is a simple application of Markov's inequality to show that the probability of the event $\{ \sup_{t\in T}Z_t \ge L\mathbb{E}\sup_{t\in T}Z_t\}$ is small, the simple geometric witnesses provide a much more refined structure on this event. Similarly, in Theorem \ref{Conj 5.7}, which is in the setting of positive selector processes, 
the role of the half-spaces is replaced by the upsets $\langle S \rangle$, and being $p$-small is an analog of admitting a cover by half-spaces with small total measure. The meaning of Theorem \ref{thm:pos-emp} can also be sketched in a similar way -- it describes explicit ``simple'' witnesses (half-spaces of the empirical measure) that cover the tail event of $\sup_{f}Z_f$. The covering perspective, as observed by Talagrand \cite{Talagrand}, also provides striking connections between these sets of questions and the study of thresholds, specifically the Kahn-Kalai conjecture. Building on the insights in the present paper, particularly the notion of minimum fragment, we obtain in \cite{KKC} the resolution of the Kahn-Kalai conjecture. 
 
Finally, we would like to mention that the ``abstract setting'' of \cite[Conjecture 5.7]{Talagrand}, which is \cite[Problem 4.2]{Talagrand06}, \cite[Conjecture 7.1]{Talagrand} and \cite[Research Problem 13.3.2]{Talagrand2}, remains open.

\mn \textbf{Reformulations.}  We will prove Theorem \ref{Conj 5.7} via the slightly more convenient equivalent reformulation below. 
As observed in \cite{Talagrand}, the following theorem is equivalent to Theorem \ref{Conj 5.7}:

\begin{theorem}\label{Conj 5.7'}
There exists $L'>0$ such that the following holds. Consider any $0<p<1$, any finite set $X$, and any family $\cF \sub 2^X$. Assume that for each $S \in \cF$ we are given a sequence $\bgl^S=(\gl^S(i))_{i \in X}$ with $\gl^S(i) \ge 0$ and
\beq{wt.lb'} \sum_{i \in S} \gl^S(i) \ge 1. \enq
Then if $\cF$ is \textbf{not} $p$-small, we have
\beq{aim'} \EE \sup_{S \in \cF} \sum_{i \in X_p} \gl^S(i) \ge 1/L'.\enq
\end{theorem}

\nin Note that we only have $\sup_{S \in \cF} \mathbb E \sum_{i \in X_p} \gl^S(i) \ge p$, so \eqref{aim'} 
suggests a nontrivial phenomenon. 
The theorem below implies Theorem \ref{Conj 5.7'}. As usual, for $m \ge 0$, we denote by ${X \choose m}$ the collection of subsets of $X$ of size $m$.

\begin{theorem}\label{Conj 5.7''}
There exists $K>0$ such that the following holds. Consider any $0<p<1$, any finite set $X$, and any family $\cF \sub 2^X$. Assume that for each $S \in \cF$ we are given a sequence $\bgl^S=(\gl^S(i))_{i \in S}$ with $\gl^S(i) \ge 0$ and
\beq{wt.lb} \sum_{i \in S} \gl^S(i) \ge 1. \enq
Suppose $\cF$ is \textbf{not} $p$-small. Then, for $\WW$ chosen uniformly at random from ${X \choose \lfloor K |X|p \rfloor}$, we have
\beq{aim} \EE \sup_{S \in \cF} \sum_{i \in S \cap \WW} \gl^S(i) \ge 10^{-11}.\enq

\end{theorem}

The derivation of Theorem \ref{Conj 5.7} from Theorem \ref{Conj 5.7'} can be found in \cite{Talagrand}. We include the simple proof of Theorem \ref{Conj 5.7'} from Theorem \ref{Conj 5.7''}, and the derivation of Theorem \ref{Conj 5.7} from Theorem \ref{Conj 5.7'} in Section \ref{red} for completeness.

The general weighted setting of Theorem \ref{Conj 5.7'} and Theorem~\ref{thm:multi-w} (the main results leading toward Theorem~\ref{Conj 5.7} and Theorem \ref{thm:pos-emp}) poses significant challenges, as one can anticipate from the statement: while the assumption on $\cF$ is inherently combinatorial, the conclusion applies to general weight functions on the sets in $\cF$. In the simplest unweighted case (where we restrict all $S\in \cF$ to have the same size $s$, and $\gl^S(i)=\frac{1}{s} \mathbb{I} (i\in S)$), our proof shares some inspiration with the argument in \cite{ALWZ, FKNP, KNP}, and the recent improvement in \cite{KKC}, although in this special case there are also alternative approaches. In particular, as we mentioned earlier, in this case (and more generally in the special case where the class of functions consists of $\{0,1\}$-valued functions), Talagrand \cite{Talagrand06} gave a nice proof using a second moment argument. The treatment of the general weighted setting (where we place no restriction on $S\in \cF$ and $\gl^S$) requires a different set of ideas. 

Firstly we work with elements at different dyadic scales, and one important idea is to choose to work with only certain ``informative'' scales. Secondly, a major difficulty of the general setting is that the weight of an element can vary with the set $S\in \cF$. In order to address this problem, we need a much more delicate and involved notion of \emph{fragments} (Definition \ref{def.fragment}), as well as the associated notion of \emph{minimum fragment} and its analysis. The use of minimum fragment (associated to a simple notion of fragment), inspired by the present work, is crucial in our recent resolution of the Kahn-Kalai conjecture \cite{KKC}. We note that while the improved analysis in \cite{KKC} using minimum fragment is not required if one wants to establish the fractional version of the Kahn-Kalai conjecture, it remains critical in our proof even if one only wants to establish the weaker fractional version of Theorem \ref{Conj 5.7} (\cite[Conjecture 6.8]{Talagrand}). In particular, the main difficulty of weight functions changing with $S\in \cF$ remains in the fractional version of Theorem \ref{Conj 5.7}, and the full strength and flexibility of our generalized notion of fragment and minimum fragment is required to handle this problem, which is how we arrive at these notions.

As discussed later in Section \ref{sec.Thm1.3}, proving Theorem \ref{thm:pos-emp}, especially in the general case where we do not impose that the underlying distribution $\nu$ is continuous, requires more ideas. In particular, we will need a version of Theorem \ref{Conj 5.7} with multiplicities, Theorem \ref{thm:multi-w}, which itself admits an equivalent reformulation, Theorem \ref{thm:multi'-w}. We emphasize that the core principles behind the proof of these results are the same as those behind the proof of Theorem \ref{Conj 5.7}. 

\mn \textbf{Organization.} The core part of the paper, the proof of Theorem~\ref{Conj 5.7''}, is contained in Section \ref{sec.proof}. The proof of Theorem \ref{thm:pos-emp} is given in Section \ref{sec.Thm1.3}, building on a version of Theorem~\ref{Conj 5.7''} with multiplicities, Theorem \ref{thm:multi'-w}. In the course, we have made no attempt to optimize absolute constants. Logarithms are in base $e$ unless otherwise specified. 

\section{Proof of Theorem \ref{Conj 5.7}}\label{sec.proof}

\subsection{Reductions between Theorem \ref{Conj 5.7}, Theorem \ref{Conj 5.7'} and Theorem \ref{Conj 5.7''}}\label{red}
In this section, we give short reductions from Theorem \ref{Conj 5.7} to Theorem \ref{Conj 5.7'} and Theorem \ref{Conj 5.7''}. 

The next derivation can be found in \cite{Talagrand} for the ``weakly $p$-small'' version of Theorem \ref{Conj 5.7}.
\begin{proof}[Proof of Theorem \ref{Conj 5.7} from Theorem \ref{Conj 5.7'}] 
Consider a collection $\Lambda$ of sequences $\bgl=(\gl_i)_{i \in X}$ with $\lambda_i \ge 0$, and with $L$ to be determined, consider the collection
\[ \cG:=\left\{S \sub X:\sup_{\bgl \in \Lambda} \sum_{i \in S} \gl_i \ge L\EE \sup_{\bgl \in \Lambda} \sum_{i \in X_{p}} \gl_i\right\}.\]
By definition of $\cG$, for each $S \in \cG$ there is $(\tau^S(i))_{i \in X} \in \Lambda$ for which
\beq{6.9} \sum_{i \in S} \tau^S(i) \ge L\EE \sup_{\bgl \in \Lambda} \sum_{i \in X_{p}} \gl_i. \enq
Now for each $S \in \cG$ define $\bgl^S$:
\[\gl^S(i)=\begin{cases}
\tau^S(i) & \mbox{ if } i \in S; \\
0 & \mbox{ otherwise. }
\end{cases}\]
Thus, for any $Y \sub X$, we have
\[\sup_{\bgl \in \Lambda} \sum_{i \in Y} \gl_i \ge \sum_{i \in Y} \tau^S(i) \ge \sum_{i \in Y} \gl^S(i)\]
and, hence, in particular
\beq{6.10} \mathbb E\sup_{\bgl \in \Lambda} \sum_{i \in X_{p}} \gl_i \ge \mathbb E \sup_{S \in \cG} \sum_{i \in X_{p}}\lambda^S(i).\enq
Assume for contradiction that $\cG$ is not $p$-small. Then by \eqref{6.9} and Theorem \ref{Conj 5.7'}, we have
\[ \EE \sup_{S \in \cG} \sum_{i \in X_{p}} \gl^S(i) \ge (L/L') \mathbb E \sup_{\bgl \in \Lambda} \sum_{i \in X_{p}} \gl_i. \]
Combining with \eqref{6.10}, we have
\[\mathbb E \sup_{\bgl \in \Lambda} \sum_{i \in X_{p}} \gl_i \ge (L/L') \mathbb E \sup_{\bgl \in \Lambda} \sum_{i \in X_{p}} \gl_i.\]
This is a contradiction for $L >L'$.
\end{proof}

\begin{proof}[Proof of Theorem \ref{Conj 5.7'} from Theorem \ref{Conj 5.7''}] 

Note that if $\cF$ is not $p$-small, then $|X|p \ge 1/2$ (since $\{\{x\}:x \in X\}$ covers $\cF$). With $N:=\max\{\lfloor |X|p \rfloor, 1\}$, let $\WW'$ be chosen uniformly at random from ${X \choose N}$. Then with $\zeta:=N /\lfloor K |X|p \rfloor$, we can think of choosing $\WW'$ as choosing $\WW$ first and then picking a $\zeta$-fraction of it. We consider two cases.

If $|X|p \ge 1$, then using the facts that $\zeta \ge 1/(2K)$ and that $\pr(|X_p| \ge \lfloor |X|p \rfloor) \ge 1/2$ (see~\cite{Lord}),
\[\begin{split}
\EE \sup_{S \in \cF} \sum_{i \in X_p} \gl^S(i) \ge \frac{1}{2}\EE \sup_{S \in \cF} \sum_{i \in S \cap \WW'} \gl^S(i) \ge\frac{1}{4K}\EE \sup_{S \in \cF} \sum_{i \in S \cap \WW} \gl^S(i) \ge \frac{1}{4 \cdot10^{11}K}.
\end{split}\]

If $1/2 \le |X|p < 1$, then using the facts that $\zeta \ge 1/K$ and that $\pr(|X_p| \ge 1) \ge 1 - (1-p)^{|X|} \ge 1 - e^{-1/2} > 1/4$, 
\[\begin{split}
\EE \sup_{S \in \cF} \sum_{i \in X_p} \gl^S(i) \ge \frac{1}{4}\EE \sup_{S \in \cF} \sum_{i \in S \cap \WW'} \gl^S(i) \ge\frac{1}{4K}\EE \sup_{S \in \cF} \sum_{i \in S \cap \WW} \gl^S(i) \ge \frac{1}{4 \cdot10^{11}K}.
\end{split}\]
Now, Theorem \ref{Conj 5.7'} follows by letting $L'=4 \cdot10^{11}K$.
\end{proof}

\subsection{Proof of Theorem \ref{Conj 5.7''}}\label{main}
Given the previous reductions, the key step in the proof of Theorem \ref{Conj 5.7} is the proof of Theorem \ref{Conj 5.7''}, which is the focus of this section and the key result in the paper containing the main insights.

In this section, $p, X, \cF,$ and  $\bgl^S$ are as in Theorem \ref{Conj 5.7''}, and $K$, a universal constant, is chosen sufficiently large to support our proof. We use $n$ for $|X|$, and $J$ and $w$ are quantities that satisfy
\[ \lfloor Knp \rfloor=Jnp=w;\]
$S,S',S''$ and $\hat S$ represent members of $\cF$, and $W \in {X \choose w}$. Finally, for $m \in \mathbb Z^+$, $[m]$ denotes $\{0,1,\ldots,m\}$.

 We say $W$ is \textit{good} if
\[\max_{S \in \cF} \sum_{i \in S \cap W} \gl^S(i) \ge 10^{-10},\]
and \textit{bad} otherwise. 
 Note that \eqref{aim} follows if a $(1/10)$-fraction of ${X \choose w}$ is good. Therefore, to prove Theorem~\ref{Conj 5.7''}, it suffices to show that
\beq{STS.ass} \mbox{if a $(9/10)$-fraction of ${X \choose w}$ is bad,}\enq
 then 
\beq{STS} \mbox{$\cF$ is $p$-small.}\enq

Before getting to the details of the proof, we first give an informal description of our overall strategy. Roughly speaking, a family $\cF$ is $p$-small if $\cF$ admits a ``cheap'' cover, where being cheap refers to the condition in \eqref{p-small}. In order to derive \eqref{STS}, we will first construct a cover of $\cF$, where the cover, $\cU=\cU(W)$, depends on the choice of $W$. We will show that the overall cost of the covers among \textit{bad} $W$'s is small, from which, combined with \eqref{STS.ass}, the existence of a cheap cover is guaranteed.

We first need some pre-processing steps on the weights $\gl^S$. 

\begin{obs}\label{reduction1}
Let $\tau=\lfloor \log_{100} n\rfloor+2$. It is sufficient to prove Theorem \ref{Conj 5.7''} assuming that, for all $S \in \cF$ and $i \in X$,
\beq{wt.ass1} \mbox{$\gl^S(i)=100^{-j}$ for some $j \in \{0,1,2,\ldots, \tau\}$ ;} \enq
with \eqref{wt.lb} weakened to
\beq{wt.lb'} \sum_{i \in S} \gl^S(i) \ge 100^{-2}.\enq
\end{obs}

\begin{proof}[Justification.]
We may first assume that 
\beq{le1} \mbox{$\gl^S(i) \le 1$ for all $S$ and $i$,}\enq
by capping larger weights at $1$. Under this assumption, for each $S$, let 
\[S_j=\{i \in S: \gl^S(i) \in [100^{-j}, 100^{-j+1})\} \quad (j=0,1,2,\ldots).\] By replacing the weights of elements in $S_j$ with $100^{-j}$, we can assume that $\gl^S(i)=100^{-j}$ for all $i \in S_j$ with \eqref{wt.lb} weakened to $\sum_{i \in S} \gl^S(i) \ge 1/100$. Finally, note that 
\[\sum_{j>\tau}\gl^S(S_j) \le |S|10^{-(\tau+1)} \le n100^{-\log_{100} n-2}=100^{-2}. \]
Thus, by removing elements in $S_j$ for $f>\tau$, we can assume that $S_j=\emptyset$ for all $j>\tau$ with \eqref{wt.lb} weakened to $\sum_{i \in S} \gl^S(i) \ge 100^{-2}$.
\end{proof}

From now on, we assume \eqref{wt.ass1}. For $S \in \cF$, we write $s_j$ for $|S_j|$, and define the \textit{profile} of $S$ to be $\sss=\sss(S)=(s_0,\ldots,s_\tau)$.
In Observation \ref{obs2} below, we use the trivial fact that
\beq{obs.spread} \mbox{a family $\cF=\{S_1,S_2,\ldots\}$ is $p$-small if $\cF'=\{S'_1, S'_2,\ldots \}$ is $p$-small and $S_i \supseteq S'_i \quad \forall i$.}\enq

\begin{obs}\label{obs2}
It is sufficient to prove Theorem \ref{Conj 5.7''} assuming that for all $S \in \cF$, in addition to \eqref{wt.ass1}, for each $j \in [\tau]$ either $s_j=0$ or
\beq{sj.lb} 100^j/(100^4(j+1)^2) \le s_j <2\cdot100^{j}; \mbox{ and}\enq
\beq{sj.power} \mbox{$s_j$ is a power of $100$} \enq
with \eqref{wt.lb'} weakened to
\beq{wt.lb''} \sum_{i \in S} \gl^S(i) \ge 100^{-4}.\enq
\end{obs}

\begin{proof}[Justification.]  First, we can greedily remove elements from each $S$ until $\gl^S(S)<2$ (note that we have \eqref{le1} under \eqref{wt.ass1}). This gives the upper bound in \eqref{sj.lb} since $2>\gl^S(S)=\sum s_j 100^{-j}$.

Next, for each $S$ and $j$, remove all elements in $S_j$ for which $s_j <100^j/(100^4(j+1)^2)$. This process reduces the weight of $S$ by at most
\[\sum_{j \ge 0} 100^j/(100^4(j+1)^2)\cdot 100^{-j}<100^{-3}.\]

Finally, for each $j$ with $|S_j|>0$, remove elements from $S_j$ to round $|S_j|$ to the largest power of $100$ that is at most $|S_j|$. The weight of $S_j$ remains at least a $1/100$ fraction of its weight before removal of elements, yielding \eqref{wt.lb''}.
\end{proof}

We say $\sss=(s_0, s_1,\ldots, s_\tau)$ is a \textit{legal profile} if each nonzero $s_j \in \sss$ satisfies
\beq{legal4} \mbox{$\max\{1, 100^j/(100^4(j+1)^2)\} \le s_j<2\cdot100^j $}\enq
and
\beq{legal5}\mbox{$s_j$ is a power of $100$.} \enq
We will also need the notion of a \textit{partial legal profile} $\sss_b:=(s_j:j \le b)$, in which each nonzero $s_j$ satisfies \eqref{legal4} and \eqref{legal5}. Given $\sss$ or $\sss_b$, we use $\cF_\sss$ for the collection of $S$'s whose profile is $\sss$, and $\cF_{\sss_b}$ for the collection of $S$'s whose partial profile is $\sss_b$. Note that if $\sss$ restricted on $\{j \le b\}$ is $\sss_b$, then $\cF_\sss \sub \cF_{\sss_b}$.

The following definition is a key notion in our proof. 

\begin{mydef}\label{def.fragment} Given $S \in \cF_\sss$ and $W \in {X \choose w}$, we say $U \sub S \setminus W$ is an $(S,W)$-fragment (with index $b$) if the following holds: there is $b \in \{-1, 0, \ldots, \tau\}$ and $S' \in \cF_{\sss_b}$ such that
\beq{fragment1}  S'_j \sub W \cup U\quad  \forall j \le b;\enq
\beq{fragment2} \sum_{j \ge b+1} \lambda^{S'}(S'_j \cap (W \cup U)) \ge .01\left[\sum_{j \ge b+1} \gl^{S'}(S'_j)-100^{-b}\left (\sum_{j \le b} s_j -|U|\right)\right]. \enq

\end{mydef}

\nin For example, $S \setminus W$ is an $(S,W)$-fragment with index $\tau$.

\begin{remark}
A similar notion of fragment is used in \cite{ALWZ, FKNP, KNP} and in the recent resolution of the Kahn--Kalai conjecture~\cite{KKC}, in which the definition of fragment is much simpler. The more delicate definition of fragment, as in Definition \ref{def.fragment}, is crucial for current setting of Theorem \ref{Conj 5.7''}.
\end{remark}

 Given $S$ and $W$, we denote by $T=T(S,W)$ the \textit{minimum} $(S,W)$-fragment, where here minimum refers to the index $b$ first (breaking ties arbitrarily), and then $|T|$ (again, breaking ties arbitrarily). We use $t=t(S,W)$ for $|T(S,W)|$. For any pair $(S,W)$, its minimum fragment $T$ and the corresponding index $b_T = b(S,W)$ are uniquely determined.

The following definition will be crucial in the construction of $\cU(W)$, a cover of $\cF$.

\begin{mydef}\label{feasible} Given $Z \sub X$, $b \in [\tau]$, $t \ge 0$, and a partial profile $\sss_b$, we say $S' \in \cF_{\sss_b}$ is $(Z,b,\sss_b,t)$-feasible if
\beq{feasible1}  S'_j \sub Z \quad \forall j \le b; \enq

\beq{feasible2} \sum_{j \ge b+1} \gl^{S'}(S'_j \cap Z) \ge .01\left[\sum_{j \ge b+1} \gl^{S'}(S'_j)-100^{-b}\left (\sum_{j \le b} s_j -t\right)\right].\enq

\end{mydef}

\nin The definitions of minimum fragment and feasibility are closely related, as shown in Propositions \ref{obs.T} and \ref{obs.T'}.

\begin{prop}\label{obs.T} Let $S\in \cF_\sss$. Let  $T$ be the minimum $(S,W)$-fragment with index $b_T$, and $t=|T|$. If $S'$ is $(W \cup T,b_T, \sss_{b_T},t)$-feasible, then 
\beq{obs.T1} (\bigcup_{j \le b_T} S'_j) \setminus W = T.\enq

\end{prop}

\begin{proof}
Note that, by the definition of feasibility, $S'$ satisfies \eqref{fragment1} and \eqref{fragment2} with $U=T$ and $b=b_T$; that $(\bigcup_{j \le {b_T}} S'_j)  \setminus W \sub T$ follows from this definition. To show $T \sub \bigcup_{j \le {b_T}}  S'_j$, first observe that by minimality of $|T|$, $T \sub S'$. Indeed, if $x \in T \setminus S'$, then replacing $T$ by $T \setminus \{x\}$, one can check that $S'$ still satisfies \eqref{fragment1} and \eqref{fragment2}. Suppose there is $x \in T \cap S'_{j_0}$ for some $j_0 \ge b_T+1$. Note that
\beq{x.wt} \gl^{ S'}(x) \le 100^{-(b_T+1)}. \enq
Then $T':=T \setminus \{ x\}$ is an $(S,W)$-fragment, since $S'$ trivially satisfies \eqref{fragment1}, and 
\[\begin{split} \sum_{j \ge b_T+1} \gl^{S'}(S'_j \cap (W \cup T')) &\ge \sum_{j \ge b_T+1} \gl^{S'}(S'_j \cap (W \cup T)) -\gl^{S'}(x)\\
& \stackrel{\eqref{feasible2}, \eqref{x.wt}}{\ge} .01\left[\sum_{j \ge b_T+1} \gl^{S'}(S'_j)-100^{-b_T}\left (\sum_{j \le b_T} s_j -|T|\right)-100^{-b_T}\right]\\
& = .01\left[\sum_{j \ge b_T+1} \gl^{S'}(S'_j)-100^{-b_T}\left (\sum_{j \le b_T} s_j -|T'|\right)\right]. \end{split}\]
This contradicts the minimality of $T$.
\end{proof}

\begin{prop}\label{obs.T'} Let $S\in \cF_\sss$. Let  $T$ be the minimum $(S,W)$-fragment with index $b_T$, and $t=|T|$. If $S'$ is $(W \cup T,b_T, \sss_{b_T},t)$-feasible, then 
\beq{obs.T3} t \ge .9\sum_{j \le b_T}s_j.\enq

\end{prop}

\begin{proof}
Let $T_{j}=S'_{j}\cap T$ and $t_j=|T_j|$ (so $t=\sum_{j \le b} t_j$ by Proposition \ref{obs.T}). We will prove the proposition by contradiction, showing that the failure of \eqref{obs.T3} violates the minimality of $T$.

First observe that, if we assume  $ t < .9\sum_{j \le b_T}s_j$, then there exists $b' \le b_T$ such that
\beq{t.le.s}
\sum_{b'\le j\le b_T}t_j 100^{-j}<.9\sum_{b' \le j \le b_T}s_{j}100^{-j}.
\enq
Indeed, if $(u_{b'}:=)  \sum_{b' \le j \le b_T} 100^{-j}(t_j-.9s_j) \ge 0$ for all $b' \le b_T$, then
\[
\sum_{j \le b_T}(t_j-.9s_{j})=\sum_{i\le b_T-1}100^{i}(u_{i}-u_{i+1})+100^{b_T}u_{b_T} = \sum_{i=1}^{b_T} (100^{i}-100^{i-1})u_i + u_0 \ge 0.
\]
Note that \eqref{t.le.s} gives, by Proposition \ref{obs.T},
\beq{lb1} \sum_{b' \le j \le b_T} \gl^{S'}(S'_j \cap W) \ge .1 \sum_{b' \le j \le b_T} \gl^{S'}(S'_j).\enq

Now we claim that
\beq{claim} \mbox{$T':=\bigcup_{j \le b'-1} T_j$ is an $(S,W)$-fragment with index $b'-1$,}
\enq
which contradicts the minimality of $T$.

\begin{proof}[Proof of \eqref{claim}]
$S'$ clearly satisfies property \eqref{fragment1}. For \eqref{fragment2},
\begin{align*}
 & \sum_{j \ge b'}\gl^{S'}(S'_j \cap (W \cup T'))\\
 &\stackrel{\eqref{obs.T1}}{=} \sum_{j \ge b_T+1}\gl^{S'}(S'_j \cap (W \cup T))+\sum_{b' \le j \le b_T}\gl^{S'}(S'_j \cap W)\\
 & \stackrel{\eqref{fragment2}, \eqref{lb1}}{ \ge} .01\left[\sum_{j \ge b_T+1}\gl^{S'}(S'_j)-100^{-b_T} \left(\sum_{j \le b_T}s_j-t\right)\right]+.1\sum_{b' \le j \le b_T}\gl^{S'}(S'_j)\\
 & = .01\sum_{j \ge b'}\gl^{S'}(S'_j)-100^{-(b_T+1)}\left(\sum_{j \le b_T}s_{j}-t\right)+.09\sum_{b' \le j \le b_T} s_j100^{-j}\\
 & \stackrel{(\dagger)}{\ge}.01\left[\sum_{j \ge b'}\gl^{S'}(S'_j)-100^{-(b'-1)}\left(\sum_{j\le b'-1}s_{j}-t'\right)\right],
\end{align*}
where $(\dagger)$ follows from the inequalities
\[.09\sum_{b' \le j \le b_T} s_j100^{-j} \ge 100^{-(b_T+1)}\sum_{b' \le j \le b_T} s_j; \mbox{ and}\]
\[(100^{-b'}-100^{-(b_T+1)})\sum_{j \le b'-1} s_j \stackrel{\eqref{obs.T1}}{\ge} (100^{-b'}-100^{-(b_T+1)})t' \ge 100^{-b'}t'-100^{-(b_T+1)}t. \qedhere\]
\end{proof} 
This completes the proof of the proposition.
\end{proof}

\begin{prop}\label{obs.B.nonempty}
If $W$ is bad, then for any $S$, we have $b_T \ge 0$.
\end{prop}

\begin{proof} If $b_T=-1$, then by Proposition \ref{obs.T} we have $T=\emptyset$. Then by \eqref{fragment2}, there is $S'$ for which
\[\gl^{S'}(S' \cap W) \ge  .01 \gl^{S'}(S')  \ge 10^{-10}, \]
which contradicts the fact that $W$ is bad.
\end{proof}

\mn

Now we construct $\cU(W)$, the promised cover of $\cF$. The following construction is valid for all $W \in {X \choose w}$, but in the end we will be only interested in bad $W$'s. 

For $W \in {X \choose w}$, define $\cU=\cU(W)$ to be
\[\cU(W):=\left\{T(S,W): S \in \cF \right\}.\]
Note that $\cU(W)$ covers $\cF$ since $T(S,W) \sub S$ for each $S \in \cF$.

The following lemma shows that if we only consider bad $W$'s, then the overall cost (in \eqref{p-small}) for the covers is cheap. Recall that $n=|X|$ and $w=Jnp$.

\begin{lemma}\label{ML}
We have
\[\sum_{\text{$W$ bad}}\,\,\sum_{U \in \cU(W)}p^{|U|} \le J^{-c}{n \choose w}\] 
for some constant $c>0$.
\end{lemma}

 Note that Lemma \ref{ML}, combined with \eqref{STS.ass}, implies that there is a bad $W$ for which
\[\sum_{U \in \cU(W)} p^{|U|} \le (10/9)J^{-c}<1/2.\]
This gives \eqref{STS} and thus concludes the proof of Theorem \ref{Conj 5.7''}.

\begin{proof}[Proof of Lemma \ref{ML}]
Let a bad $W$ be given. For each $b$, a legal partial profile $\sss_b$ and $t \ge 0$, let
\[\cG_W(\sss_b,t)=\{S:S \in \cF_{\sss_b}, t(S,W)=t, b(S,W)=b\};\]
and
\[\cU_W(\sss_b,t)=\{T(S,W):S \in \cG_W(\sss_b,t)\}.\]
Then we have, with $n_b:= \sum_{j \le b}s_j$ for given $\sss_b$,
\beq{final} \sum_{\text{$W$ bad}}\,\,\sum_{U \in \cU(W)} p^{|U|} \le \sum_{b \ge 0} \,\, \sum_{\sss_b \text{ legal}} \,\, \sum_{.9n_b \le t \le n_b} \,\, \sum_{\text{$W$ bad}} \,\, \sum_{U \in  \cU_W(\sss_b,t)} p^{|U|},\enq
noting that the range $b \ge 0$ follows from Proposition \ref{obs.B.nonempty}, and the range of $t$ is from Propositions \ref{obs.T} and \ref{obs.T'}.

Given $b\ge 0$, $\sss_b$ and $t$, we bound  $\sum_{\text{$W$ bad}} \sum_{U \in  \cU_W(\sss_b,t)} p^{|U|}$ as follows:

\begin{enumerate}[Step 1.]
\item Pick $Z:=W \cup T$. Since $|Z|=w+t$ (note $W$ and $T$ are always disjoint), the number of possibilities for $Z$ is at most (recalling $w=Jpn$)
\[{n \choose w+t} = {n \choose w} \cdot \prod_{j=0}^{t-1}\frac{n-w-j}{w+j+1} \le {n\choose w} (Jp)^{-t}.\]

\item Given $Z$, by the definition of fragment, there must exist a choice of $(Z, b, \sss_b, t)$-feasible $\hat S$. Make a choice of $\hat S$ (arbitrarily) such that it only depends on $Z$. 
Then Proposition \ref{obs.T} enables us to specify $T$ as a subset of $\bigcup_{j\le b}\hat S_j$, whose number of possibilities is at most $2^{n_b}$.
\end{enumerate}

Therefore, the right-hand side of \eqref{final} is at most
\[\begin{split} \sum_{b \ge 0} \,\, \sum_{\sss_b \text{ legal}} \,\, \sum_{.9n_b \le t \le n_b} \,\, \sum_{\text{$W$ bad}} \,\, \sum_{U \in  \cU_W(\sss_b,t)} p^t & \le  \sum_{b \ge 0} \,\, \sum_{\sss_b \text{ legal}} \,\, \sum_{.9n_b \le t \le n_b} {n \choose w}J^{-t}2^{n_b}\\
& \le {n \choose w} \sum_{b \ge 0}\,\, \sum_{\sss_b \text{ legal}} (J/4)^{-.9n_b}.\\
\end{split}\]
Notice that, by \eqref{legal4} and \eqref{legal5}, the number of possibilities for legal $\sss_b$ is at most
\[ \prod_{j \le b} \log_{100}(2\cdot100^4(j+1)^2).\]
On the other hand, using the lower bound on $s_j$ in \eqref{legal4},
\[n_b\ge \sum_{j \le b} \max\{1, 100^j/(100^4(j+1)^2)\}.\]
Therefore, with
\beq{aj}a_j:=\log_{100}[2\cdot 100^4(j+1)^2]\cdot(J/4)^{-.9\max\{1, 100^j/(100^4(j+1)^2)\}},\enq
we have
\[\sum_{b \ge 0}\sum_{\sss_b \text{ legal}} (J/4)^{-.9n_b}\le \prod_{j \in [\tau]}(1+a_j)-1 \le \exp\left(\sum_{j \in [\tau]}a_j\right)-1 \le J^{-c}\]
for some positive constant $c$, assuming that $J$ is chosen sufficiently large. 
\end{proof}

\section{Proof of Theorem \ref{thm:pos-emp}}\label{sec.Thm1.3}
\subsection{A version of Theorem \ref{Conj 5.7} with multiplicities}\label{sec.3.1}
As a key tool in the proof of Theorem \ref{thm:pos-emp}, we need a version of Theorem \ref{Conj 5.7} for multisets. Given a finite set $X$, a \textit{multiset} over $X$ is a function $m:X\to \mathbb{N}_0$ (so $m(i)$ is the ``multiplicity'' of $i \in X$). Let $\cM(X)$ denote the collection of multisets over $X$. The \textit{size} of a multiset is given by $|m|:=\sum_{i\in X}m(i)$. Let $\cM_N(X)$ be the collection of multisets over $X$ of size $N$. For multisets $S$ and $T$, we say that $T$ is a \textit{subset} of $S$ and denote $T\subseteq S$ if $T(x)\le S(x)$ for all $x$. Given a collection $\cF$ of multisets over $X$, we say that $\cG$ is a \textit{cover} of $\cF$ if for every $S\in \cF$, there exists $T\in \cG$ satisfying $T\subseteq S$. 
Given two multisets $S$ and $T$, we define the \textit{union} $S\vee T$ as the multiset with $(S \vee T)(i) = \max(S(i),T(i))$ and the \textit{intersection} $S \wedge T$ as the multiset with $(S\wedge T)(i)=\min(S(i),T(i))$. We also define the \emph{disjoint union} $S+T$ as the multiset with $(S+T)(i)=S(i)+T(i)$, and the \textit{set difference} $S\setminus T$ as the multiset with $(S\setminus T)(i) = \max(S(i)-T(i),0)$. For an element $x\in X$ and a multiset $S\in \cM(X)$, we say that $x\in S$ if $S(x)>0$ and $x\notin S$ if $S(x)=0$.

We next give an overview of our proof of Theorem \ref{thm:pos-emp}, which requires several additional ideas. First, we need the following multiset and weighted generalization of Theorem \ref{Conj 5.7}. 

\begin{theorem}\label{thm:multi-w}
There exist $L>0$ with the following property. Consider any finite set $X$, and any collection $\Lambda$ of sequences $\bgl=(\gl_i)_{i \in X}$ with $\gl_i \ge 0$. Let $\mu:X\to \mathbb{R}_{\ge 0}$ be a probability measure on $X$. Let $N$ be a positive integer. Let $\WW$ be a random multiset of size $N$ over $X$ where for each multiset $W$ of size $N$, 
\[
\pr[\WW=W] = N! \prod_{x\in X}\frac{\mu(x)^{W(x)}}{W(x)!}.
\]
Then the family 
\beq{multi-col}
\left\{S \in \cM(X):\sup_{\bgl \in \Lambda} \sum_{i\in X} S(i) \gl_i \ge L\EE \sup_{\bgl \in \Lambda} \sum_{i \in X} \WW(i)\gl_i\right\}
\enq
admits a cover $\cG$ with 
\beq{eq.multi-w-cost}
\sum_{G\in \cG} \prod_{x\in X}\frac{(eN\mu(x))^{G(x)}}{G(x)!} \le 1/2.
\enq
\end{theorem}

Note that for any positive integer $N$, the random multiset $\WW$ induced by $N$ independent samples from $\mu$ has distribution 
\[
\pr[\WW=W] = N! \prod_{x\in X}\frac{\mu(x)^{W(x)}}{W(x)!}.
\]
In particular, for any positive integer $N$, the above distribution is a valid probability distribution. 

In the next subsection, we discuss the reduction of Theorem \ref{thm:pos-emp} to Theorem \ref{thm:multi-w}. We first replace the distribution $\nu$ underlying the empirical process by an appropriate distribution $\mu$ over a finite set. For an appropriate large integer $P$, we partition the Polish space $\mathbb{T}$ into subsets $B_1 \dots B_P$ such that for each function $f\in \cF$, $f$ is approximately constant on each set. We then define the measure $\mu$ by setting $\mu(h)=\nu(B_h)$ for each $h\in [P]$. For each function in $\cF$, we construct a sequence $\lambda$ on $[P]$ by choosing an arbitrary element $x_h\in B_h$ and set $\lambda(h)=f(x_h)$. Let $\Lambda$ be the collection of sequences $\lambda$ corresponding to functions $f\in \cF$. 

To illustrate the connection to selector processes, observe that in the case $\nu$ is continuous, i.e. $\nu(\{x\})=0$ for any $x\in \mathbb{T}$, we can guarantee a partition as above in which $\mu(h) \in [1/(2P),1/P]$ for all $h$. In this case, the empirical process given by $N$ independent samples of $\nu$ is closely connected to the selector process on $[P]$ with parameter $p=N/P$  indexed by $\Lambda$. In a previous version of this paper, we give a proof of this special case following a simplified version of our strategy for Theorem \ref{thm:pos-emp} utilizing multisets. %
The follow-up work of Bednorz, Martynek and Meller \cite{BMM} shows that this case can be treated without resorting to multiset generalization. In particular, in this case, one observes that for $P$ large a small cover can be constructed for the small probability event that an element appears more than once in $N$ samples of $\mu$. 

In the general case where $\nu$ has atoms, this observation no longer holds: %
repeated appearance of elements in samples from $\nu$ is no longer a rare event. Multisets arise and as such, the multiset generalization, Theorem \ref{thm:multi-w}, of Theorem \ref{Conj 5.7} is natural in this context. We prove Theorem \ref{thm:pos-emp} via the discretization of $\nu$ into $\mu$, and translate the desired conclusion into the multiset covering in Theorem \ref{thm:multi-w}. %
We emphasize that this entire plan presents no additional conceptual difficulty. Indeed, we shall present the entire proof of Theorem \ref{thm:multi-w} in parallel with the proof of Theorem \ref{Conj 5.7''}. One shall see that the only difference is the bookkeeping of combinatorial factors that arise due to the atomic nature of $\nu$.

As will be evident in the proof of Theorem \ref{thm:pos-emp} from Theorem \ref{thm:multi-w}, it suffices in (\ref{multi-col}) to consider only multisets $S \in \cM_N(X)$ of fixed size $N$. However, it turns out that it is not important to keep track of this restriction, and in fact, in proving Theorem \ref{thm:multi-w}, we will go through Theorem \ref{thm:multi'-w}, which works instead with random multisets of size $KN$ for an appropriate constant $K$. Thus, we find it more convenient and general to work with $S\in \cM(X)$ in (\ref{multi-col}). 

In Section \ref{sec:discont}, we prove Theorem \ref{thm:pos-emp} assuming Theorem \ref{thm:multi-w}, which requires careful construction of events $E_{g,t}$ in Theorem \ref{thm:pos-emp} from the cover $\cG$ in Theorem \ref{thm:multi-w}. Finally, in Section \ref{app:pos-emp-w}, we give the full details of the proof of Theorem \ref{thm:multi-w}, which follows along the line of the proof of Theorem \ref{Conj 5.7''}.%

\subsection{Proof of Theorem \ref{thm:pos-emp} assuming Theorem \ref{thm:multi-w}}\label{sec:discont}
Since $\EE[\sup_{f\in \cF}Z_f] \in (0,\infty)$, by renormalizing, we can assume that 
\beq{normalization'} \EE[\sup_{f\in \cF}Z_f] = 1. \enq Let $\epsilon>0$ be sufficiently small. Let $\cF = \{f_1 \dots f_M\}$. 

Let $U = \max_{i\le M} \sup f_i$, and note from (\ref{normalization'}) that $U\ge 1$. Partition $\mathbb{T}$ into sets of the form $I(k_1\dots k_M):=\bigcap_{i=1}^{M} f_i^{-1}([k_i\epsilon,(k_i+1)\epsilon))$, where $k_i \in [0,\lfloor U/\epsilon\rfloor]$ are integers. 
Define a new probability distribution $\mu$ on $\Omega:=\{(k_1\dots k_M):k_i \in [0,\lfloor U/\epsilon\rfloor]\}$ with $\mu(k_1\dots k_M)= \nu(I(k_1\dots k_M))$ and define $X_1 \dots X_N$ as independent samples from the distribution $\mu$. For each $i\in [M]$, define $\lambda^i:\Omega\to \mathbb{R}_{\ge 0}$ by $\lambda^i(k_1\dots k_M)=k_i\epsilon$, and let $\Lambda = \{\lambda^i:i\in [M]\}$. Note that for each $y\in I(k_1\dots k_M)$, 
\beq{bound-eps}
|f_i(y)-\lambda^i(k_1\dots k_M)|\le \epsilon.
\enq 

Given independent samples $X_1\dots X_N$ from $\mu$, the multiset given by $X_1\dots X_N$ is a multiset $\WW$ of size $N$ with distribution
\[
\pr[\WW=W] = N! \prod_{x\in \Omega}\frac{\mu(x)^{W(x)}}{W(x)!}.
\]

By Theorem \ref{thm:multi-w}, we can find a collection $\cG \subseteq \cM(\Omega)$ with 
\beq{cost-b}
\sum_{G\in \cG} \prod_{x\in \Omega}\frac{(eN\mu(x))^{G(x)}}{G(x)!} \le 1/2,
\enq
and furthermore $\cG$ covers 
\beq{col}
\left\{S \in \cM(\Omega):\sup_{\bgl \in \Lambda} \sum_{x\in \Omega} S(x) \gl_x \ge L\EE \sup_{\bgl \in \Lambda} \sum_{x \in \Omega} \WW(x)\gl_x\right\}.
\enq
Notice that $\frac{(eN\mu(x))^{G(x)}}{G(x)!} \ge \left(\frac{eN\mu(x)}{G(x)}\right)^{G(x)}$. As such, by removing elements $x\in G$ in which $N\mu(x)/G(x)>1$ from each $G\in \cG$ (after which $\cG$ remains a cover of (\ref{col}) and the left hand side of (\ref{cost-b}) decreases), we can assume without loss of generality that 
\beq{le-1}
N\mu(x)/G(x) \le 1 \textrm{ for all $G\in \cG$ and $x\in G$}.
\enq 

For each $G\in \cG$, let $\tilde{g}_G:\Omega\to \mathbb{R}$ be the function $\tilde{g}_G(x)=\log (N\mu(x)/G(x))$ for $x\in G$ and $\tilde{g}_G(x)=0$ otherwise. Let $\tilde{t}_G=\sum_{x\in G} G(x)\log (N\mu(x)/G(x))$. Then, using (\ref{le-1}), for any $(X_1\dots X_N)$ for which the corresponding multiset $W$ contains $G$, we have that 
\[
\prod_{i: X_i\in G} \frac{N\mu(X_i)}{G(X_i)} = \prod_{x\in G} \left(\frac{N\mu(x)}{G(x)}\right)^{W(x)} \le \prod_{x\in G} \left(\frac{N\mu(x)}{G(x)}\right)^{G(x)},
\]
so $(X_1\dots X_N)$ is contained in the event 
\begin{equation}\label{eq:tildeH}
\cH_{\tilde{g}_G,\tilde{t}_G}=\left\{\sum_{i=1}^{N}\tilde{g}_G(X_i) \le \tilde{t}_G\right\}.
\end{equation} 

We compute 
\begin{align}
\EE\left[\exp\left(-\sum_{i=1}^{N} \tilde{g}_G(X_i)\right)\right] &= \left(\mathbb{E}\left[\exp(-\tilde{g}_G(X_1))\right]\right)^{N}\nonumber\\
&= \left(\sum_{x\in G} \mu(x) \frac{G(x)}{N\mu(x)} + \sum_{x\notin G} \mu(x)\right)^{N} \nonumber\\
&\le (1+|G|/N)^{N}. \label{eq:E}
\end{align}
Note that $-\tilde{g}_G(x) \ge 0$ for all $x$. Hence, by Markov's Inequality and (\ref{eq:E}),
\begin{align}
\pr[\cH_{\tilde{g}_G,\tilde{t}_G}] &= \pr\left[\exp\left(-\sum_{i=1}^{N} \tilde{g}_G(X_i)\right) \ge \exp(-\tilde{t}_G)\right] \nonumber \\
&\le \exp(\tilde{t}_G)  (1+|G|/N)^{N} \nonumber \\
&= (1+|G|/N)^{N} \prod_{x\in G}\left(\frac{N\mu(x)}{G(x)}\right)^{G(x)}. 
\end{align}
We also have $N\log(1+|G|/N) \le |G|$, and thus 
\begin{align}
\pr[\cH_{\tilde{g}_G,\tilde{t}_G}] &\le (1+|G|/N)^{N} \prod_{x\in G}\left(\frac{N\mu(x)}{G(x)}\right)^{G(x)} \nonumber \\
&\le \prod_{x\in G}e^{G(x)} \left(\frac{N\mu(x)}{G(x)}\right)^{G(x)} \nonumber \\
&\le \prod_{x\in G} \frac{(eN\mu(x))^{G(x)}}{G(x)!}, \label{eq.final-H}
\end{align}
where in the last inequality we use the estimate $G(x)! \le G(x)^{G(x)}$.

From (\ref{eq.final-H}) and (\ref{eq.multi-w-cost}), we have 
\begin{equation}
\sum_{G\in \cG}\pr[\cH_{\tilde{g}_G,\tilde{t}_G}] \le 1/2.
\end{equation}

Finally, for $y\in \mathbb{T}$, define $\pi:\mathbb{T}\to \Omega$ by $\pi(y)=(k_1\dots k_M)$ for $y\in I(k_1 \dots k_M)$. Observe that for independent samples $(Y_1 \dots Y_N)$ from $\nu$, defining $X_j = \pi(Y_j)$ for $j\in [N]$ and $\mathbf{W}$ the multiset given by $X_1 \dots X_N$, then $(X_1 \dots X_N)$ are independent samples from $\mu$. Furthermore, by (\ref{bound-eps}), for all $i\in [M]$,
\[
\left|Z_{f_i}(Y_1 \dots Y_N) - \frac{1}{N}\sum_{x\in \Omega} \mathbf{W}(x)\lambda^i(x)\right|\le \epsilon,
\] 
and
\[
\left|\mathbb{E}\left[\sup_{i\in [M]}Z_{f_i}(Y_1 \dots Y_N)\right] - \frac{1}{N}\mathbb{E}\left[\sup_{i\in [M]}\sum_{x\in \Omega} \mathbf{W}(x)\lambda^i(x)\right]\right|\le \epsilon.
\]
As such, assuming that $\sup_{f\in \cF} Z_{f}(Y_1 \dots Y_N) \ge 2L\mathbb{E}\sup_{f\in \cF}Z_f$, recalling (\ref{normalization'}), for $\epsilon$ sufficiently small, we have that $\WW$ is contained in (\ref{col}) and hence covered by $\mathcal{G}$. Thus, by (\ref{eq:tildeH}), 
\[
\left\{\sup_{f\in \cF}Z_f\ge 2L\mathbb{E} \sup_{f\in \cF}Z_f\right\} \subseteq \bigcup_{G\in \cG} \pi^{-1}(\cH_{\tilde{g}_G,\tilde{t}_G}).
\]
Noting that $\pi^{-1}(\cH_{\tilde{g}_G,\tilde{t}_G})$ is exactly the same as the event $E_{g,t}$ for $g=-N\tilde{g}_G\circ\pi$ and $t=-\tilde{t}_G$ for each $G\in \cG$, this yields Theorem \ref{thm:pos-emp} (upon modifying the value of the constant $L$). 

\begin{remark}\label{rm:sym}
Theorem \ref{thm:multi-w} also directly implies the following statement. For any $c>0$, there exists $L>0$ such that the following holds. We say that a subset of $\mathbb{T}^k$ is symmetric if it is invariant under coordinate permutations. For each $k\ge 1$, there exists a symmetric set $V_{k} \subseteq \mathbb{T}^k$ such that, if $(Y_1 \dots Y_N)$ is contained in the event $\left\{\sup_{f\in \cF} Z_f \ge L \EE\sup_{f\in \cF} Z_f\right\}$, then there exists $i_1< \ldots <i_k$ such that $(Y_{i_1},\ldots,Y_{i_k}) \in V_{k}$. Furthermore, for each $k\ge 1$, $\pr\left((Y_1 \dots Y_k)\in V_{k}\right) \le \frac{1}{2} (ck/N)^k$. 

In particular, the above statement follows from Theorem \ref{thm:multi-w} by taking $V_k = \bigcup_{G\in \cG, |G|=k}V_k(G)$, where $V_k(G)$ is the collection of tuples $(y_1 \dots y_k)\in \mathbb{T}^{k}$ whose corresponding multiset is equal to $G$. 

This implies (a stronger version of) the result in \cite{Talagrand06} on positive empirical processes in the special case where the class of functions only involves indicator functions of sets, and generalizes this result to the setting of general nonnegative functions.  
\end{remark}

\subsection{Proof of Theorem \ref{thm:multi-w}}\label{app:pos-emp-w}
Theorem \ref{thm:multi-w} follows from Theorem \ref{thm:multi'-w} below. 
\begin{theorem}\label{thm:multi'-w}
There exist a positive integer $K$ with the following property. Consider any finite set $X$, and any family of multisets $\cF \sub \cM(X)$. Let $\mu:X\to \mathbb{R}_{\ge 0}$ be a probability measure on $X$. Let $N$ be a positive integer. Let $\WW$ be a random multiset of size $KN$ over $X$ where for each multiset $W$ of size $KN$, 
\[
\pr[\WW=W] = {(KN)!} \prod_{x\in X}\frac{\mu(x)^{W(x)}}{W(x)!}.
\]

Assume that for each $S \in \cF$ we are given a sequence $\bgl^S=(\gl^S(i))_{i \in X}$ with $\gl^S(i) \ge 0$ and
\beq{wt.lb'-multi} \sum_{i \in X} S(i) \gl^S(i) \ge 1. \enq
Then if there is no cover $\cG$ for $\cF$ with 
\[
\sum_{G\in \cG} \prod_{x\in X}\frac{(eN\mu(x))^{G(x)}}{G(x)!} \le 1/2,
\]
then we have
\beq{aim'-multi} \EE \sup_{S \in \cF} \sum_{i \in X} \WW(i) \gl^S(i) \ge 10^{-11}.\enq
\end{theorem}

\begin{proof}[Proof of Theorem \ref{thm:multi-w} from Theorem \ref{thm:multi'-w}]

By Theorem \ref{thm:multi'-w}, under the assumptions in Theorem \ref{thm:multi'-w}, a random multiset $\tilde{\WW}$ over $X$ of size $KN$ with distribution
\[
\pr[\tilde{\WW}=W] = {(KN)!}\prod_{x\in X}\frac{(\mu(x))^{W(x)}}{W(x)!}
\]
satisfies 
\[
\mathbb{E}\sup_{S\in \cF} \sum_{i\in X}\tilde{\WW}(i)\lambda^S(i) \ge 10^{-11}. 
\]
Let $\WW$ be a random multiset of size $N$ with distribution 
\[
\pr[\WW=W]=N!\prod_{x\in X}\frac{(\mu(x))^{W(x)}}{W(x)!}.
\]
Note that $\tilde{\WW}$ has the distribution of a multiset induced by a tuple of $KN$ independent samples from $\mu$, and $\WW$ has the distribution of a multiset induced by a tuple of $N$ independent samples from $\mu$. Hence, we have 
\beq{lower-e-W}
\mathbb{E}\sup_{S\in \cF} \sum_{i\in X}{\WW}(i)\lambda^S(i) \ge \frac{N}{KN}\mathbb{E}\sup_{S\in \cF} \sum_{i\in X}\tilde{\WW}(i)\lambda^S(i)\ge \frac{1}{K}\cdot 10^{-11}.
\enq
Theorem \ref{thm:multi-w} then follows as in the proof of Theorem \ref{Conj 5.7} from Theorem \ref{Conj 5.7'}. In particular, letting 
\[
\cF = \left\{S \in \cM(X):\sup_{\bgl \in \Lambda} \sum_{i\in X} S(i) \gl_i \ge L\EE \sup_{\bgl \in \Lambda} \sum_{i \in X} \WW(i)\gl_i\right\},
\]
for each $S\in \cF$, there is $(\tau^S(i))_{i\in X} \in \Lambda$ with 
\beq{lower-tau} \sum_{i\in X} S(i) \tau^S(i) \ge L\EE \sup_{\bgl \in \Lambda} \sum_{i \in X} \WW(i)\gl_i.\enq 
Define $\bgl^S$ by $\lambda^S(i)=\tau^S(i)$ if $i\in S$ and $\lambda^S(i)=0$ otherwise, then 
\beq{contra-1} \EE \sup_{\bgl \in \Lambda} \sum_{i \in X} \WW(i)\gl_i \ge \EE \sup_{S\in \cF} \sum_{i \in X} \WW(i)\gl^S(i).\enq
However, assume for the sake of contradiction that $\cF$ does not admit a cover satisfying (\ref{eq.multi-w-cost}), then by (\ref{lower-tau}) and (\ref{lower-e-W}), 
\beq{contra-2} \EE \sup_{S\in \cF} \sum_{i \in X} \WW(i)\gl^S(i) \ge \frac{1}{K}10^{-11}L\EE \sup_{\bgl \in \Lambda} \sum_{i \in X} \WW(i)\gl_i. \enq
For $L$ sufficiently large, (\ref{contra-1}) and (\ref{contra-2}) yield the desired contradiction. 
\end{proof}

The proof of Theorem \ref{thm:multi'-w} follows along the main ideas behind the proof of Theorem \ref{Conj 5.7''}. Below we give the full details to make transparent the parallel with the proof of Theorem \ref{Conj 5.7''}. We emphasize that only direct adaptations are needed and there is no conceptual difficulty in following the argument upon having the proof of Theorem \ref{Conj 5.7''}. %

We follow the definitions in Section \ref{sec.3.1}. In the following, denote $M=KN$, and $S,S',\hat{S}$ are always elements of $\cF$. We say that $W\in \cM_M(X)$ is \textit{good} if $\max_{S\in \cF} \sum_{i} W(i)\gl^S(i) \ge 10^{-10}$, and bad otherwise.

By the same processing steps as in Section \ref{sec.proof} (Observations \ref{reduction1} and \ref{obs2}), we can assume that $\gl^S(i) = 100^{-j}$ for some $j \in \mathbb N_0$, 
with the right hand side of \eqref{wt.lb'-multi} weakened to $100^{-1}$. Denote by $S_j$ the sub-multiset of $S$ consisting of all the elements $i$ with $\gl^S(i)=100^{-j}$. Let $s_j = |S_j|$. Then we can assume that for each $j \in \mathbb N_0$, 
either $s_j = 0$ or $100^{j-4}/(j+1)^2 \le s_j < 2\cdot 100^j$ and $s_j$ is a power of $100$, with the right hand side of \eqref{wt.lb'-multi} weakened to $100^{-3}$. Let $\tau=\max\{j:s_j \ne 0\}$. Note that $\tau <\infty$ since $\cF$ is finite.
The following definition is a straightforward generalization of Definition \ref{def.fragment} to the multiset case. 

\begin{mydef}\label{def.fragment.multi} Given $S \in \cF_\sss$ and $W \in \cM_M(X)$, we say $U\subseteq S\setminus W$ is an $(S,W)$-fragment (with index $b$) if the following holds: there is $b \in \{-1, 0, \ldots, \tau\}$ and $S' \in \cF_{\sss_b}$ such that
\beq{fragment1-multi}  S'_j \sub W + U\quad  \forall j \le b;\enq
\beq{fragment2-multi} \sum_{j \ge b+1} \lambda^{S'}(S'_j \wedge (W + U)) \ge .01\left[\sum_{j \ge b+1} \gl^{S'}(S'_j)-100^{-b}\left (\sum_{j \le b} s_j - |U|\right)\right]. \enq

\end{mydef}

Given $S$ and $W$, we denote by $T=T(S,W)$ the \textit{minimum} $(S,W)$-fragment, where here minimum refers to the index $b$ first (breaking ties arbitrarily), and then $|T|$ (again, breaking ties arbitrarily). We use $t=t(S,W)$ for $|T(S,W)|$. For any pair $(S,W)$, its minimum fragment $T$ and the corresponding index $b_T = b(S,W)$ are uniquely determined. 

Definition \ref{feasible} can be directly adapted to the current setting. 

\begin{mydef}\label{feasible-multi} Given $Z\in \cM(X)$, $b \in [\tau]$, $t \ge 0$, and a partial profile $\sss_b$, we say $S' \in \cF_{\sss_b}$ is $(Z,b,\sss_b,t)$-feasible if
\beq{feasible1-multi}  S'_j \sub Z \quad \forall j \le b; \enq

\beq{feasible2-multi} \sum_{j \ge b+1} \gl^{S'}(S'_j \wedge Z) \ge .01\left[\sum_{j \ge b+1} \gl^{S'}(S'_j)-100^{-b}\left (\sum_{j \le b} s_j -t\right)\right].\enq

\end{mydef}

\nin The following propositions generalize Propositions \ref{obs.T}, \ref{obs.T'} and \ref{obs.B.nonempty}. Their proofs are direct adaptations of the proofs in Section \ref{sec.proof} with little changes. We have included the details for completeness. 

\begin{prop}\label{obs.T-multi} For $S \in \cF_{\sss}$, let  $T$ be the minimum $(S,W)$-fragment with index $b_T$, and $t=|T|$. If $S'$ is $(W + T,b_T, \sss_{b_T},t)$-feasible, then 
\beq{obs.T1-multi} \left(\bigvee_{j \le b_T} S'_j\right) \setminus W = T.\enq

\end{prop}
\begin{proof}
Note that, by the definition of feasibility, $S'$ satisfies \eqref{fragment1-multi} and \eqref{fragment2-multi} with $U=T$ and $b=b_T$; that $(\bigvee_{j \le {b_T}} S'_j)  \setminus W \sub T$ follows from this definition. 

To show $T \sub (\bigvee_{j \le {b_T}}  S'_j) \setminus W$, first observe that by minimality of $|T|$, $T \sub S' \setminus W$. Indeed, if $x \in T \setminus (S'\setminus W)$, then replacing $T$ by $T \setminus \{x\}$, one can check that $S'$ still satisfies \eqref{fragment1-multi} and \eqref{fragment2-multi}. Suppose there is $x \in T \wedge S'_{j_0}$ for some $j_0 \ge b_T+1$. Note that
\beq{x.wt-w} \gl^{ S'}(x) \le 100^{-(b_T+1)}. \enq
Then $T':=T \setminus \{ x\}$ is an $(S,W)$-fragment, since $S'$ trivially satisfies \eqref{fragment1-multi}, and 
\[\begin{split} \sum_{j \ge b_T+1} \gl^{S'}(S'_j \wedge (W + T')) &\ge \sum_{j \ge b_T+1} \gl^{S'}(S'_j \wedge (W + T)) -\gl^{S'}(x)\\
& \stackrel{\eqref{feasible2-multi}, \eqref{x.wt-w}}{\ge} .01\left[\sum_{j \ge b_T+1} \gl^{S'}(S'_j)-100^{-b_T}\left (\sum_{j \le b_T} s_j -|T|\right)-100^{-b_T}\right]\\
& = .01\left[\sum_{j \ge b_T+1} \gl^{S'}(S'_j)-100^{-b_T}\left (\sum_{j \le b_T} s_j -|T'|\right)\right]. \end{split}\]
This contradicts the minimality of $T$.
\end{proof}

\begin{prop}\label{obs.T'-multi} For $S \in \cF_{\sss}$, let $T$ be the minimum $(S,W)$-fragment with index $b_T$, and $t=|T|$. If $S'$ is $(W + T,b_T, \sss_{b_T},t)$-feasible, then 
\beq{obs.T3-multi} t \ge .9\sum_{j \le b_T}s_j.\enq
 
\end{prop}

\begin{proof}
Let $T_{j}=S'_{j}\wedge T$ and $t_j=|T_j|$ (so $t=\sum_{j \le b_T} t_j$ by Proposition \ref{obs.T-multi}). We will prove the proposition by contradiction, showing that the failure of \eqref{obs.T3-multi} violates the minimality of $T$.

First observe that, if we assume  $ t < .9\sum_{j \le b_T}s_j$, then there exists $b' \le b_T$ such that
\beq{t.le.s-multi}
\sum_{b'\le j\le b_T}t_j 100^{-j}<.9\sum_{b' \le j \le b_T}s_{j}100^{-j}.
\enq
Indeed, if $(u_{b'}:=)  \sum_{b' \le j \le b_T} 100^{-j}(t_j-.9s_j) \ge 0$ for all $b' \le b_T$, then
\[
\sum_{j \le b_T}(t_j-.9s_{j})=\sum_{i\le b_T-1}100^{i}(u_{i}-u_{i+1})+100^{b_T}u_{b_T} = \sum_{i=1}^{b_T} (100^{i}-100^{i-1})u_i + u_0 \ge 0.
\]
Note that \eqref{t.le.s-multi} gives, by Proposition \ref{obs.T-multi},
\beq{lb1-multi} \sum_{b' \le j \le b_T} \gl^{S'}(S'_j \wedge W) \ge .1 \sum_{b' \le j \le b_T} \gl^{S'}(S'_j).\enq

Now we claim that
\beq{claim-multi} \mbox{$T':=\bigvee_{j \le b'-1} T_j$ is an $(S,W)$-fragment with index $b'-1$,}
\enq
which contradicts the minimality of $T$.

\begin{proof}[Proof of \eqref{claim-multi}]
$S'$ clearly satisfies property \eqref{fragment1-multi}. For \eqref{fragment2-multi},
\begin{align*}
 & \sum_{j \ge b'}\gl^{S'}(S'_j \wedge (W + T'))\\
 &\stackrel{\eqref{obs.T1-multi}}{=} \sum_{j \ge b_T+1}\gl^{S'}(S'_j \wedge (W + T))+\sum_{b' \le j \le b_T}\gl^{S'}(S'_j \wedge W)\\
 & \stackrel{\eqref{fragment2-multi}, \eqref{lb1-multi}}{ \ge} .01\left[\sum_{j \ge b_T+1}\gl^{S'}(S'_j)-100^{-b_T} \left(\sum_{j \le b_T}s_j-t\right)\right]+.1\sum_{b' \le j \le b_T}\gl^{S'}(S'_j)\\
 & = .01\sum_{j \ge b'}\gl^{S'}(S'_j)-100^{-(b_T+1)}\left(\sum_{j \le b_T}s_{j}-t\right)+.09\sum_{b' \le j \le b_T} s_j100^{-j}\\
 & \stackrel{(\dagger)}{\ge}.01\left[\sum_{j \ge b'}\gl^{S'}(S'_j)-100^{-(b'-1)}\left(\sum_{j\le b'-1}s_{j}-t'\right)\right],
\end{align*}
where $(\dagger)$ follows from the inequalities
\[.09\sum_{b' \le j \le b_T} s_j100^{-j} \ge 100^{-(b_T+1)}\sum_{b' \le j \le b_T} s_j; \mbox{ and}\]
\[(100^{-b'}-100^{-(b_T+1)})\sum_{j \le b'-1} s_j \stackrel{\eqref{obs.T1-multi}}{\ge} (100^{-b'}-100^{-(b_T+1)})t' \ge 100^{-b'}t'-100^{-(b_T+1)}t. \qedhere\]
\end{proof} 
This completes the proof of the proposition.
\end{proof}

\begin{prop}\label{obs.B.nonempty-multi}
If $W$ is bad, then for any $S$, we have $b_T \ge 0$.
\end{prop}

\begin{proof} If $b_T=-1$, then by Proposition \ref{obs.T-multi} we have $T=\emptyset$. Then by \eqref{fragment2-multi}, there is $S'$ for which
\[\gl^{S'}(S' \wedge W) \ge  .01 \gl^{S'}(S')  \ge 10^{-10}, \]
which contradicts the fact that $W$ is bad.
\end{proof}

For $W \in \cM_M(X)$, define $\cU=\cU(W)$ to be
\[\cU(W):=\left\{T(S,W): S \in \cF \right\}.\]
Note that $\cU(W)$ covers $\cF$ since $T(S,W) \sub S$ for each $S \in \cF$.

We will also need the following simple observation.
\begin{prop}\label{obs.max}
For $S\in \cF_{\sss}$, let $T$ be the minimum $(S,W)$-fragment with index $b_T$, and $t=|T|$. If $S'$ is $(W+T,b_T,\sss_{b_T},t)$-feasible, then $W(x)< S'(x)$ for all $x$ with $T(x)>0$.  
\end{prop}
\begin{proof}
Suppose that there is $x$ with $T(x)>0$ and $W(x)\ge S'(x)$. Then for $T'$ with $T'(y)=T(y)$ for $y\ne x$ and $T'(x)=T(x)-1$, we have that $S'$ is $(W+T',b_T,\sss_{b_T},t-1)$-feasible, contradicting minimality of $T$. 
\end{proof}

Choose an appropriate $c>0$ and sufficiently large $J_0>0$ so that, with (\textit{cf.} \eqref{aj})
\beq{aj'}a_j:=\log_{100}[2\cdot 100^4(j+1)^2]\cdot(e^{-4}J_0)^{-.9\max\{1, 100^j/(100^4(j+1)^2)\}},\enq
we have
\beq{Jc-w} \exp\left(\sum_{j \ge 0}a_j\right)-1 \le J_0^{-c}.\enq
Now we choose $K$ sufficiently large so that $K \ge J_0$. 

\begin{lemma}\label{ML-multi-w}
We have 
\[\sum_{\text{$W$ bad}}\pr[\WW=W]\sum_{U \in \cU(W)} \prod_{x\in X}\frac{(eN\mu(x))^{U(x)}}{U(x)!} \le J_0^{-c}\]
with $c$ and $J_0$ as in \eqref{Jc-w}.
\end{lemma}

\begin{proof}[Proof of Lemma \ref{ML-multi-w}]
For a given bad $W\in \cM_{M}(X)$ (recalling that $M=KN$), integer $b$, a partial profile $\sss_b$ and $t \ge 0$, let
\[\cG_W(\sss_b,t)=\{S:S \in \cF_{\sss_b}, t(S,W)=t, b(S,W)=b\};\]
and
\[\cU_W(\sss_b,t)=\{T(S,W):S \in \cG_W(\sss_b,t)\}.\]
Then we have, with $n_b := \sum_{j \le b}s_j$ for given $\sss_b$,
\begin{align}\label{final-multi'} &\sum_{\text{$W$ bad}}\pr[\WW=W]\sum_{U \in \cU(W)} \prod_{x\in X}\frac{(eN\mu(x))^{U(x)}}{U(x)!} \nonumber\\& \le \sum_{b \ge 0} \,\, \sum_{\sss_b \text{ legal}} \,\, \sum_{.9n_b \le t \le n_b} \,\, \sum_{\text{$W$ bad}} \,\, \sum_{U \in  \cU_W(\sss_b,t)} \pr[\WW=W] \prod_{x\in X}\frac{(eN\mu(x))^{U(x)}}{U(x)!},\end{align}
noting that the range $b \ge 0$ follows from Proposition \ref{obs.B.nonempty-multi}, and the range of $t$ is from Propositions \ref{obs.T-multi} and \ref{obs.T'-multi}.

Fix $\sss_b$ ($b \ge 0$) and $t$. Let $\ZZ$ be a random multiset in $\cM_{M+t}(X)$ with 
\[
\pr[\ZZ=Z] = (M+t)! \prod_{x\in X}\frac{\mu(x)^{Z(x)}}{Z(x)!}.
\]

Let $\cM_{(M,t)}(X)$ be the collection of $Z\in \cM(X)$ that can be written as $W+T$ for $W\in \cM_M(X)$ which is bad and $T\in \cU_W(\sss_b,t)$ has size $t$ (so $|Z|=M+t$). By the definition of fragment, for $Z\in \cM_{(M,t)}(X)$, there must exist a choice of $(Z, b, \sss_b, t)$-feasible $\hat S$. Make a choice of $\hat S=\hat S(Z)$ (arbitrarily) such that it only depends on $Z$. Then Proposition \ref{obs.T-multi} shows that $T \subseteq \bigvee_{j\le b}\hat S_j$. In particular, we have 
\beq{inj}
\sum_{\text{$W$ bad}} \,\, \sum_{U \in  \cU_W(\sss_b,t)}  \,\, \pr[\WW=W] \prod_{x\in X}\frac{(eN\mu(x))^{U(x)}}{U(x)!} \le \sum_{Z\in \cM_{(M,t)}(X)}\,\,\sum_{T\subseteq \bigvee_{j\le b}\hat S_j} \,\,\pr[\WW=Z\setminus T] \prod_{x\in X} \frac{(eN\mu(x))^{T(x)}}{T(x)!}.
\enq 
It remains to bound the right hand side:
\begin{align*}
&\sum_{Z\in \cM_{(M,t)}(X)}\,\, \sum_{T\subseteq \bigvee_{j\le b}\hat S_j} \pr[\WW=Z\setminus T]  \prod_{x\in X}\frac{(eN\mu(x))^{T(x)}}{T(x)!} \\
&=\sum_{Z\in \cM_{(M,t)}(X)}\,\,\sum_{T\subseteq \bigvee_{j\le b}\hat S_j} \pr[\ZZ=Z] \frac{\pr[\WW=Z\setminus T]}{\pr[\ZZ=Z]} \cdot \prod_{x\in X}\frac{(eN\mu(x))^{T(x)}}{T(x)!}\\
&=\sum_{Z\in \cM_{(M,t)}(X)}\pr[\ZZ=Z] \,\, \sum_{T\subseteq \bigvee_{j\le b}\hat S_j}  \,\,\frac{M!}{(M+t)!} \prod_{x\in X}\left(\mu(x)^{-T(x)}\frac{Z(x)!}{(Z(x)-T(x))!}\right) \cdot N^t\prod_{x\in X}\frac{(e\mu(x))^{T(x)}}{T(x)!}\\
&\le \sum_{Z\in \cM_{(M,t)}(X)}\pr[\ZZ=Z] \frac{N^t}{M^t}e^{t}\sum_{T\subseteq \bigvee_{j\le b}\hat S_j}  \,\, \prod_{x\in X}\binom{Z(x)}{T(x)}.
\end{align*}
Here, in the last inequality, we have used the trivial bound $M!/(M+t)!\le M^{-t}$. Furthermore, by Proposition \ref{obs.max}, we have $\hat S(x)>W(x)$ for all $x$ with $T(x)>0$. In particular, for $x$ with $T(x)>0$, we have $Z(x)=W(x)+T(x)<\hat S(x)+T(x)$, and thus 
\[
\binom{Z(x)}{T(x)}\le \binom{\hat S(x)+T(x)}{T(x)} \le 2^{\hat S(x)+T(x)}\le 2^{2\hat S(x)}.
\]
Hence, recalling that $M=KN$, 
\begin{align*}
&\sum_{Z\in \cM_{(M,t)}(X)}\pr[\ZZ=Z] \frac{N^t}{M^t}e^{t}\sum_{T\subseteq \bigvee_{j\le b}\hat S_j}  \,\, \prod_{x\in X}\binom{Z(x)}{T(x)} \\
&\le \sum_{Z\in \cM_{(M,t)}(X)}\pr[\ZZ=Z] (e/K)^{t}\,\,\sum_{T\subseteq \bigvee_{j\le b}\hat S_j} \,\, \prod_{x\in  \bigvee_{j\le b}\hat S_j} 2^{2\hat S(x)}\\
&\le  \sum_{Z\in \cM_{(M,t)}(X)}\pr[\ZZ=Z] (e/K)^{t} \cdot 2^{2n_b}\cdot 2^{n_b}\\
&=  (e/K)^{t} \cdot 2^{3n_b}.
\end{align*}
Here, in the second inequality, we used that the number of $T \subseteq \bigvee_{j\le b}\hat S_j$ is at most $2^{n_b}$ and in the last step we used that $\sum_{Z\in \cM_{(M,t)}(X)}\pr[\ZZ=Z]\le1$. Finally, we have
\[\begin{split} \sum_{b \ge 0} \,\, \sum_{\sss_b \text{ legal}} \,\, \sum_{.9n_b \le t \le n_b} \,\, \sum_{\text{$W$ bad}} \,\, \sum_{U \in  \cU_W(\sss_b,t)}  \,\, \pr[\WW=W] \prod_{x\in X}\frac{(eN\mu(x))^{U(x)}}{U(x)!} & \le  \sum_{b \ge 0} \,\, \sum_{\sss_b \text{ legal}} \,\, \sum_{.9n_b \le t \le n_b} (e/K)^{t}\cdot 2^{3n_b}\\
& \le \sum_{b \ge 0} \,\, \sum_{\sss_b \text{ legal}} (e^{-4}J_0)^{-.9n_b}.\\
\end{split}\]

The rest of the proof proceeds identically to the proof of Theorem \ref{Conj 5.7''}. As in \eqref{legal4} and \eqref{legal5}, the number of possibilities for legal $\sss_b$ is at most
\[ \prod_{j \le b} \log_{100}(2\cdot100^4(j+1)^2).\]
On the other hand, using the lower bound on $s_j$ in \eqref{legal4},
\[n_b\ge \sum_{j \le b} \max\{1, 100^j/(100^4(j+1)^2)\}.\]
Therefore, with our choice of $a_j$ from (\ref{aj'}), 
we have
\[\sum_{b \ge 0}\sum_{\sss_b \text{ legal}} (e^{-4}J_0)^{-.9n_b}\le \prod_{j \in [\tau]}(1+a_j)-1 \le \exp\left(\sum_{j \in [\tau]}a_j\right)-1 \le J_0^{-c}\]
for our choice of $c$ and $J_0$ in \eqref{Jc-w}. 
\end{proof}

Theorem \ref{thm:multi'-w} follows immediately from Lemma \ref{ML-multi-w}.
\begin{proof}[Proof of Theorem \ref{thm:multi'-w}]
Since $\cF$ does not have any cover $\cG$ with 
\[
\sum_{U \in \cG} \prod_{x\in X} \frac{(eN\mu(x))^{G(x)}}{G(x)!} \le 1/2,
\]
we have for all bad $W$ that 
\[
\sum_{U \in \cU(W)} \prod_{x\in X}\frac{(eN\mu(x))^{U(x)}}{U(x)!} > 1/2. 
\]
Thus, Lemma \ref{ML-multi-w} implies that 
\[
\mathbb{P}[\WW \textrm{ is bad}] \le 2J_0^{-c}. 
\]
Hence, 
\[
\mathbb{E} \sup_{S\in \cF}\sum_{i\in X}\WW(i)\lambda^{S}(i) \ge (1-2J_0^{-c})10^{-10} \ge 10^{-11}. 
\]
\end{proof}

\section*{Acknowledgements}
We would like to thank Michel Talagrand for raising the question on general positive empirical processes and for many helpful discussions. We are grateful to David Conlon, Amir Dembo, Jacob Fox, and Jeff Kahn for their support and useful comments. We would also like to thank the anonymous referees for insightful comments that improve the paper and prompt us to realize an oversight in a previous version of the paper. Huy Tuan Pham is supported by a Two Sigma Fellowship, a Clay Research Fellowship and a Stanford Science Fellowship. Jinyoung Park is supported by NSF grants DMS-2153844 and DMS-2324978.


\begin{thebibliography}{AA}

\bibitem{ALWZ}
R. Alweiss, S. Lovett, K. Wu, and J. Zhang, 
\emph{Improved bounds for the sunflower lemma,}
Ann. of Math. (2) \textbf{194} (2021), no. 3, 795--815.

\bibitem{BLatala} W. Bednorz and R. Lata\l a,
\emph{On the boundedness of Bernoulli processes,} 
Ann. Math. (2) \textbf{180}(3) (2014),
1167--1203.

\bibitem{BM} W. Bednorz and R. Martynek,
\emph{A note on infinitely divisible processes,}
Ann. Probab. \textbf{50}(1) (2022), 397--417.

\bibitem{BMM} W. Bednorz, R. Martynek and R. Meller,
\emph{The suprema of selector processes with the application to positive infinitely divisible processes,}
preprint (2022), arXiv:2212.14636.

\bibitem{FKNP}
K. Frankston, J. Kahn, B. Narayanan, and J. Park,
\emph{Thresholds versus fractional expectation-thresholds,}
Ann. of Math. (2) \textbf{194} (2021), no. 2, 475--495.


\bibitem{KNP}
J. Kahn, B. Narayanan, and J. Park,
\emph{The threshold for the square of a Hamilton cycle,}
Proc. Math. Am. Soc. \textbf{149}(8) (2021), 3201--3208. 


\bibitem{Lord}
N. Lord,
\emph{Binomial averages when the mean is an integer,}
Math. Gaz. \textbf{94} (2010), 331--332.

\bibitem{KKC}
J. Park and H. T. Pham,
\emph{A proof of the Kahn--Kalai Conjecture,}
to appear in J. Am. Math. Soc. (2023). 

\bibitem{SS}
S. Shalev-Shwartz and S. Ben-David, 
\emph{Understanding machine learning: from theory to algorithms,}
Cambridge university press, 2014.

\bibitem{S}
S. M. Srivastava,
\emph{A Course on Borel Sets,} Graduate Texts in Mathematics, \textbf{180}, Springer-Verlag, New York, 1998.

\bibitem{TalagrandCB}
M. Talagrand,
\emph{Regularity of gaussian processes,}
Acta Math. \textbf{159} (1987), 99--149.

\bibitem{Talagrand0}
M. Talagrand,
\emph{Donsker classes and random geometry,}
Ann. Probab. \textbf{15}(4) (1987), 1327--1338.

\bibitem{TalagrandCBs}
M. Talagrand,
\emph{A simple proof of the majorizing measure theorem,}
Geom. Funct. Anal. \textbf{2} (1992), 118--125.

\bibitem{Talagrand1}
M. Talagrand,
\emph{The generic chaining,}
Springer Monographs in Mathematics, Springer-Verlag, Berlin, 2005.

\bibitem{Talagrand06}
M. Talagrand,
\emph{Selector processes on classes of sets,}
Probab. Theory Relat. Fields \textbf{135} (2006), 471--486.

\bibitem{Talagrand}
M. Talagrand,
\emph{Are many small sets explicitly small?,}
{Proceedings of the 2010 ACM International Symposium on Theory of Computing}
(2010), 13--35.

\bibitem{Talagrand2}
M. Talagrand,
\emph{Upper and Lower Bounds for Stochastic Processes,}
Ergebnisse der Mathematik und ihrer Grenzgebiete. 3. Folge (A Series of Modern Surveys in Mathematics book series) vol. \textbf{60}, Springer, 2021.

\bibitem{Talagrand-per}
M. Talagrand, personal communication, 2022.

\bibitem{S}
S. van de Geer, 
\emph{Empirical Processes in M-estimation,} 
Cambridge Series in Statistical and Probabilistic Mathematics, Cambridge University Press, 2000.



\end{thebibliography}
\end{document}